\documentclass[12pt]{article}

\usepackage{amssymb}
\usepackage{graphicx}
\usepackage{amscd}
\usepackage{amsmath}
\usepackage{amsthm}
\usepackage{amsfonts}
\usepackage{latexsym}

\begin{document}

\theoremstyle{plain}
\newtheorem{theorem}{Theorem}
\newtheorem{corollary}[theorem]{Corollary}
\newtheorem{lemma}[theorem]{Lemma}
\newtheorem{proposition}[theorem]{Proposition}

\theoremstyle{definition}
\newtheorem{definition}[theorem]{Definition}
\newtheorem{example}[theorem]{Example}
\newtheorem{conjecture}[theorem]{Conjecture}

\theoremstyle{remark}
\newtheorem{remark}[theorem]{Remark}

\begin{center}
\vskip 1cm{\LARGE\bf Some identities involving Appell polynomials}
\vskip 1cm
\large
Miloud Mihoubi\\
\small
USTHB, Faculty of Mathematics, RECITS Laboratory, Algiers, Algeria\\
{\tt mmihoubi@usthb.dz or miloudmihoubi@gmail.com}\\
\large
Said Taharbouchet\\
\small
USTHB, Faculty of Mathematics, RECITS Laboratory, Algiers, Algeria\\
{\tt staharbouchet@usthb.dz} \\
\ \\
\end{center}

\vskip .2 in

\begin{abstract}
In this paper, by the classical umbral calculus method, we
establish identities involving the Appell polynomials and extend some existing identities.
\end{abstract}


\vskip .2 in


\section{Introduction}
Let $\left( a_{n}\right) _{n\geq 0}$ be a sequence of real numbers such that $a_{0}=1,$ let $\textbf{A}^{n}$
be the umbra defined by \ $\textbf{A}^{n}:=a_{n}$ and let $\left(
f_{n}\left( x\right)\right) _{n\geq 0}\ $ be a sequence of Appell polynomials defined by \cite%
{appel}:%
\begin{equation*}
\label{1}\sum_{n\geq 0}f_{n}\left( x\right) \frac{t^{n}}{n!}=F\left( t\right) \exp
\left( xt\right) =\exp \left( \left(\textbf{A}+x\right) t\right) ,
\end{equation*}%
where $F\left( t\right) =1+\sum_{n\geq 1}a_{n}\frac{t^{n}}{n!}=\exp \left( \textbf{A} t\right).$
So,  $f_{n}\left( x\right) $ admits the umbral representation
\begin{equation*}
\label{2} f_{n}\left( x\right) =\left( \mathbf{A}+x\right) ^{n}.
\end{equation*}%
The Appell polynomials appear in many areas of mathematics including special functions, analysis, combinatorics and number theory. Numerous interesting properties for these polynomials can be found in the literature, see for example \cite{appel,car,mac,rio,sri,sri1}. Recently, many authors took interest on symmetric identities involving Bernoulli polynomials such as the works of Kaneto \cite{kan}, Gessel \cite{gess} and He et al. \cite{he}. The present work is motivated by the work of Di Crescenzo et al. \cite{cre} on umbral calculus, the work of Gessel \cite{gess} on some applications of the classical umbral calculus, the work of He et al. \cite{he} on some symmetric identities and the work of Pita et al. \cite{Ruiz} on some identities for Bernoulli polynomials  and Benyattou et al. \cite{ben} on some applications of Bell umbra. We use the umbral calculus method to deduce symmetric identities involving Appell polynomials  and generalize some existing identities on Bernoulli and Euler polynomials.

\section{Appell polynomials via classical umbral calculus}

The following proposition gives a general symmetric identity for Appell polynomials and generalizes Theorem 1.1 of He et al. \cite{he}.
\begin{proposition}
\label{P0}Let $n,m$ be non-negative integers. There holds%
\begin{equation*}
\label{3}\sum_{k=0}^{n}\binom{n}{k}y^{n-k}f_{m+k}\left( x\right) =\sum_{k=0}^{m}%
\binom{m}{k}\left( -y\right) ^{m-k}f_{n+k}\left( x+y\right) .
\end{equation*}
\end{proposition}

\begin{proof}
By the umbral representation $f_{n}\left( x\right) =\left( \mathbf{A}+x\right) ^{n}$ we get on the one hand
\begin{eqnarray*}
\left( \mathbf{A}+x+y\right) ^{n}\left( \mathbf{A}+x\right) ^{m} &=&
\left( \mathbf{A}+x+y\right) ^{n}\left( \mathbf{A}+x+y-y\right) ^{m} \\
&=&\sum_{k=0}^{m}\binom{m}{k}\left( -y\right) ^{m-k}\left( \mathbf{A}%
+x+y\right) ^{n+k} \\
&=&\sum_{k=0}^{m}\binom{m}{k}\left( -y\right) ^{m-k}f_{n+k}\left( x+y\right)
\end{eqnarray*}
and on the other hand, we have

\begin{equation*}
\left( \mathbf{A}+x+y\right) ^{n}\left( \mathbf{A}+x\right) ^{m}=
\sum_{k=0}^{n}\binom{n}{k}y^{n-k}\left( \mathbf{A}+x\right) ^{m+k} =
\sum_{k=0}^{n}\binom{n}{k}y^{n-k}f_{m+k}\left( x\right).
\end{equation*}
Hence, the two expressions of $\left( \mathbf{A}+x+y\right) ^{n}\left( \mathbf{A}+x\right) ^{m}$ give the desired identity.
\end{proof}

\noindent By definition of the sequence $\left(f_{n}\left( x\right)\right) _{n\geq 0}\ $  we get
$D_{x}^{p}f_{n}\left( x\right) =\left( n\right)_{p}f_{n-p}\left( x\right),$
where $D_{x}^{p}=\frac{d^{p}}{dx^{p}}, \left( x\right)_{p}:=x(x-1)\cdots (x-p+1)$ if $p\geq1$ and $\left( x\right)_{0}:=1.$ So, by derivation $p$ times the two sides of the identity of Proposition \ref{P0} to respect to $x$  we obtain

\begin{corollary}
\label{CC}Let $n,m,p$ be non-negative integers such that $p\leq \min(n,m).$ There holds%
\begin{equation*}
\label{4}\sum_{k=0}^{n}\binom{n}{k}\binom{m+k}{p}y^{n-k}%
f_{m-p+k}\left( x\right) =\sum_{k=0}^{m}\binom{m}{k}\binom{n+k}{p}\left( -y\right)
^{m-k}f_{n-p+k}\left( x+y\right) .
\end{equation*}%
\end{corollary}

\noindent Let now $\left(f_{n}^{(\alpha )}\left(x\right)\right) _{n\geq 0}$ be the sequence of Appell polynomials defined, for any real number $\alpha$, by \cite%
{appel}:%
\begin{equation*}
\label{5}\sum_{n\geq 0}f_{n}^{(\alpha )}\left( x\right) \frac{t^{n}}{n!}=\left(F\left( t\right)\right) ^{\alpha }\exp
\left( xt\right) \ \ \texttt{with} \ \ f_{n}^{(1 )}\left( x\right):=f_{n}\left( x\right).
\end{equation*}%

\noindent The following lemma gives some properties of the sequence $(f_{n}^{\left(  \alpha  \right)  }\left(  x\right) )$  and has interesting generalizations on existing identities.

\begin{lemma}
\label{LL}For any real number $\alpha$  there hold%
\begin{align*}
f_{n}^{\left(  \alpha \right)  }\left( \textbf{A}+x\right)&=f_{n}^{\left(  \alpha +1 \right)  }\left( x\right), \\
\left(\alpha +1\right)\left( \textbf{A}+x\right) f_{n}^{\left( \alpha \right) }\left( \textbf{A}+x\right)
&=f_{n+1}^{\left( \alpha +1\right) }\left( x\right) +\alpha xf_{n}^{\left( \alpha +1\right) }\left( x\right).
\end{align*}
\end{lemma}

\begin{proof}
Replace $x$ by $\textbf{A}+x$ in the above generating function to obtain%
\begin{align*}
\sum_{n\geq0}f_{n}^{\left(  \alpha \right)  }\left( \textbf{A}+x\right)  \frac{t^{n}}{n!}
& =\left(F\left( t\right)\right) ^{\alpha }e^{\left( \textbf{A}+x\right)  t}\\
& =\left(F\left( t\right)\right)^{\alpha}\sum_{n\geq0}\left(\textbf{A}+x\right)  ^{n}\frac{t^{n}}{n!}\\
& =\left(F\left( t\right)\right)^{\alpha}\sum_{n\geq0}f_{n}\left(x\right)  \frac{t^{n}}{n!} \\
& =\left(F\left( t\right)\right)^{\alpha+1}e^{xt}
\\ &=\sum_{n\geq0}f_{n}^{\left(  \alpha+1\right)  }\left(  x\right)
\frac{t^{n}}{n!},%
\end{align*}
from which the first identity follows.
\\ For $\alpha=-1$ the second identity is obvious and for $\alpha\neq -1$ we have
\begin{align*}
&\sum_{n\geq 0}\left( \textbf{A}+x\right) f_{n}^{\left( \alpha \right) }\left(
\textbf{A}+x\right) \frac{t^{n}}{n!}  \\&=\left( F\left( t\right) \right) ^{\alpha
}\left( \textbf{A}+x\right) \exp \left( \left( \textbf{A}+x\right) t\right)  \\
&=\left( F\left( t\right) \right) ^{\alpha }\sum_{n\geq 0}\left( \textbf{A}+x\right)
^{n+1}\frac{t^{n}}{n!} \\
&=\left( F\left( t\right) \right) ^{\alpha }\sum_{n\geq 0}f_{n+1}\left(
x\right) \frac{t^{n}}{n!} \\ &=\left( F\left( t\right) \right) ^{\alpha }D_{t}\left( F\left( t\right)
\exp \left( xt\right) \right) \\
& =\frac{1}{\alpha +1}D_{t}\left( \left( F\left( t\right) \right) ^{\alpha
+1}\exp \left( xt\right) \right) +\frac{\alpha x}{\alpha +1}\left( F\left( t\right)
\right) ^{\alpha +1}\exp \left( xt\right)\\
&=\sum_{n\geq 0}\left( \frac{1}{\alpha +1}f_{n+1}^{\left( \alpha +1\right)
}\left( x\right) +\frac{\alpha x}{\alpha +1}f_{n}^{\left( \alpha +1\right) }\left(
x\right) \right) \frac{t^{n}}{n!}
\end{align*}
which implies the second identity.
\end{proof}

\begin{remark} The second identity of Lemma \ref{LL} can be extended as:
\begin{align*}
\left( \alpha +2\right) \left( \alpha +1\right) \left( \textbf{A}+x\right)
^{2}&f_{n}^{\left( \alpha \right) }\left( \textbf{A}+x\right)  =f_{n+2}^{\left(
\alpha +2\right) }\left( x\right) -2xf_{n+1}^{\left( \alpha +2\right)
}\left( x\right) +x^{2}f_{n}^{\left( \alpha +2\right) }\left( x\right)
\\ & +2\left( \alpha +2\right) xf_{n+1}^{\left( \alpha +1\right) }\left(
x\right) +\left( \alpha +2\right)\left( \alpha -1\right) x^{2}f_{n}^{\left( \alpha +1\right) }\left(x\right) .
\end{align*}
\end{remark}

\noindent However, Corollary \ref{CC} can be generalized via Lemma \ref{LL}, as follows.
\begin{corollary}
Let $\alpha$ be a real number and let $n,m,p$ be non-negative integers such that $p\leq \min(n,m).$ There holds%
\begin{align*}
\sum_{k=0}^{n}\binom{n}{k}&\dbinom{m+k}{p}y^{n-k}f_{m-p+1+k}^{\left(
\alpha \right) }\left( x\right)\\ & =
\sum_{k=0}^{m}\binom{m}{k}\dbinom{n+k}{p}\left( -y\right) ^{m-k}\left(
f_{n-p+1+k}^{\left( \alpha \right) }\left( x+y\right) -y
f_{n-p+k}^{\left( \alpha \right) }\left( x+y\right)\right).
\end{align*}%
\end{corollary}

\begin{proof}
From Corollary \ref{CC} we have%
\begin{align*}
\sum_{k=0}^{n}\binom{n}{k}\dbinom{m+k}{p}y^{n-k}x
f_{m-p+k}^{\left( \alpha-1 \right) }\left( x\right) =\sum_{k=0}^{m}\binom{m}{k}%
\dbinom{n+k}{p}\left( -y\right) ^{m-k}x f_{n-p+k}^{\left(
\alpha-1 \right) }\left( x+y\right) .
\end{align*}%
To obtain the desired identity, replace $x$ by $\textbf{A}+x$ and apply Lemma \ref{LL}, after that simplify the obtained identity by Corollary \ref{CC}.
\end{proof}

\noindent Other results are given by the following propositions.
\begin{proposition}
\label{P3} Let $n$ be a non-negative integer, let $\alpha$ be a real number and
let $(u_{k}), (v_{k})$ be two sequences of real numbers. If%
\begin{equation*}
\sum_{k=0}^{n}U\left( n,k\right) f_{k}\left( x+u_{k}\right) =\sum_{k=0}^{n}V\left(
n,k\right) (x+v_{k})^{k}
\end{equation*}%
for some real sequences $\left( U\left( n,k\right) ;0\leq k\leq n\right) $
and $\left( V\left( n,k\right) ;0\leq k\leq n\right) ,$ then%
\begin{align*}
\sum_{k=0}^{n}U\left( n,k\right) &f_{k}^{\left( \alpha \right) }\left(
x+u_{k}\right) =\sum_{k=0}^{n}V\left(n,k\right) f_{k}^{\left( \alpha
-1\right) }\left( x+v_{k}\right), \\
\alpha \sum_{k=0}^{n}U\left( n,k\right)& \left( f_{k+1}^{\left( \alpha
+1\right) }\left( x+u_{k}\right) -u_{k}f_{k}^{\left( \alpha +1\right)
}\left( x+u_{k}\right) \right)  \\
&=\sum_{k=0}^{n}V\left(n,k\right) \left( \left( \alpha +1\right)
f_{k+1}^{\left( \alpha \right) }\left( x+v_{k}\right) -\left( x+\left(
\alpha +1\right) v_{k}\right) f_{k}^{\left( \alpha \right) }\left(
x+v_{k}\right) \right).
\end{align*}%
\end{proposition}

\begin{proof}
On the one hand, we prove that
$f_{n}^{\left(  \alpha \right)  }\left(  x\right)  $ is a polynomial in $\alpha $ of degree at most $n.$ Indeed, by definition,  $f_{n}^{\left(  \alpha \right)  }\left(  x\right)$ can be written as
$f_{n}^{\left(  \alpha \right)  }\left(  x\right)   =\sum_{k=0}^{n}\binom
{n}{k}g_{k}\left(  \alpha \right)  x^{n-k},$
where $\left(  g_{n}\left(  \alpha \right)  \right)  _{n\geq 0}\ $\ is a sequence of
binomial type with exponential generating function $\left(F\left(t\right)\right) ^{\alpha
}$. The known relation
$g_{n}\left(  \alpha \right)=  \sum_{k=0}^{n}\mathcal{B}_{n,k}\left(g_{j}(1)\right)  \left(\alpha \right)  _{k}$ \cite{com,rom} proves the first part of this proof,
where $\mathcal{B}_{n,k}\left(  x_{j}\right):=\mathcal{B}_{n,k}\left(  x_{1},x_{2},\ldots \right) $
are the partial Bell polynomials \cite{com,miho,miho1}. On the other hand, let $Q$ be the polynomial:
\begin{equation*}
Q\left(  \alpha \right)=\sum_{k=0}^{n}U(n,k)f_{k}^{\left(  \alpha \right)  }\left(  x+u_{k}\right)-\sum_{k=0}^{n}V(n,k)f_{k}^{\left(  \alpha -1\right)  }\left(  x+v_{k}\right).
\end{equation*}
Since $f_{n}^{\left(  0\right)  }\left(  x\right)=x^{n},$ the hypothesis shows that $Q\left(  1 \right)=0.$ By replacing $x$ by $\textbf{A}+x$ in $Q\left(  1 \right)$  we obtain, in virtu of Lemma \ref{LL}, $Q\left(  2 \right)=0.$ So, by the same process we can state that $Q\left( s \right)=0$ for all non-negative integer $s.$
So, the polynomial $Q$ has infinity roots which state that necessarily $Q\left(  \alpha \right)=0$ for all real number $\alpha.$ Hence, the first identity of the proposition follows.
Multiply the obtained identity by $x$ and change $x$ by $\textbf{A}+x$ to obtain via Lemma \ref{LL}:%
\begin{align*}
&\alpha \sum_{k=0}^{n}U\left( n,k\right) \left( f_{k+1}^{\left( \alpha
+1\right) }\left( x+u_{k}\right) +\alpha \left( x+u_{k}\right) f_{k}^{\left(
\alpha +1\right) }\left( x+u_{k}\right) -\left( \alpha +1\right)
u_{k}f_{k}^{\left( \alpha \right) }\left( x+u_{k}\right) \right) \\ &=
\left( \alpha +1\right)\times \\ & \ \ \ \ \ \ \ \sum_{k=0}^{n}V\left( n,k\right) \left(
f_{k+1}^{\left( \alpha \right) }\left( x+v_{k}\right) +\left( \alpha
-1\right) \left( x+v_{k}\right) f_{k}^{\left( \alpha \right) }\left(
x+v_{k}\right) -\alpha v_{k}f_{k}^{\left( \alpha \right) }\left(
x+v_{k}\right) \right).
\end{align*}%
So, use the first identity of this proposition to simplify this last one.%
\end{proof}

\noindent An application of proposition \ref{P3} on Abel's identity is given by the following corollary.

\begin{corollary}
Let $\alpha ,\beta ,q$ be real numbers. Then, there holds%
\begin{equation*}
\label{C1} f_{n}^{\left( \alpha +\beta \right) }\left( x+y\right) =\sum_{k=0}^{n}\binom{%
n}{k}\left( f_{k}^{\left( \alpha \right) }\left( x-qk\right)
+qkf_{k-1}^{\left( \alpha \right) }\left( x-qk\right) \right) f_{n-k}^{\left(
\beta \right) }\left( y+qk\right) .
\end{equation*}
\end{corollary}

\begin{proof}
The Abel's identity%
\begin{align*}
\left( x+y\right) ^{n}=\sum_{k=0}^{n}\binom{n}{k}\left( \left(
x-qk\right) ^{k}+qk\left( x-qk\right) ^{k-1}\right) \left( y+qk\right) ^{n-k}
\end{align*}%
can be written as%
\begin{equation*}
f_{n}^{\left( 0\right) }\left( x+y\right) =\sum_{k=0}^{n}\binom{n}{k}\left( f_{k}^{\left( 0\right) }
\left( x-qk\right) +qkf_{k-1}^{\left(0\right) }\left( x-qk\right)\right)
f_{n-k}^{\left(0\right) }\left( y+qk\right)
\end{equation*}%
which implies for fixed $x,$ via Proposition \ref{P3}, the identity%
\begin{equation*}
f_{n}^{\left( \beta \right) }\left( x+y\right) =\sum_{k=0}^{n}\binom{n}{k}%
\left( f_{k}^{\left( 0\right) }\left( x-qk\right) +qkf_{k-1}^{\left(
0\right) }\left( x-qk\right) \right) f_{n-k}^{\left( \beta \right)
}\left( y+qk\right) ,
\end{equation*}%
and this implies, for fixed $y,$ the desired identity.
\end{proof}

\begin{example}
Let $\mathbb{B}_{n}^{\left(  \alpha \right)  }\left(  x\right)$ be the Bernoulli polynomials
of order $\alpha$ and $\mathbb{E}_{n}^{\left(  \alpha \right)  }\left(  x\right)$ be the Euler polynomials
of order $\alpha$ defined by  \cite{mag}: $\mathbb{B}_{n}^{\left(1\right)  }\left(  x\right)  =\mathbb{B}_{n}\left(  x\right)$,   $\mathbb{E}_{n}^{\left(1\right)  }\left(  x\right)  =\mathbb{E}_{n}\left(  x\right)$ and
$$
\sum_{n\geq0}\mathbb{B}_{n}^{\left(  \alpha \right)  }\left(  x\right)  \frac{t^{n}}%
{n!}=\left(  \frac{t}{e^{t}-1}\right)  ^{\alpha}e^{xt} ,  \ \
\sum_{n\geq 0}\mathbb{E}_{n}^{\left( \alpha \right) }\left( x\right) \frac{t^{n}}{n!}%
=\left( \frac{2}{e^{t}+1}\right) ^{\alpha }e^{xt} .
$$
Then, for any real number $\alpha,$ the identities given in \cite[Thms. 1,2]{xia} by%
\begin{align*}
\sum_{k=0}^{n}\binom{2n+1}{2k}\mathbb{E}_{2k}\left( x\right)  &=x^{2n+1}-\left(
x-1\right) ^{2n+1}, \\
\sum_{k=1}^{n}\binom{2n}{2k-1}\mathbb{E}_{2k-1}\left( x\right)  &=x^{2n}-\left(
x-1\right) ^{2n}, \\
\sum_{k=0}^{n}\binom{2n+1}{2k}\mathbb{B}_{2k}\left( x\right)  &=\left( n+\frac{1}{2}%
\right) \left( x^{2n}+\left( x-1\right) ^{2n}\right) , \\
\sum_{k=1}^{n}\binom{2n}{2k-1}\mathbb{B}_{2k-1}\left( x\right)  &=n\left(
x^{2n-1}+\left( x-1\right) ^{2n-1}\right)
\end{align*}%
imply, via Proposition \ref{P3}, the following identities
\begin{align*}
\sum_{k=0}^{n}\binom{2n+1}{2k}\mathbb{E}_{2k}^{\left( \alpha \right) }\left( x\right)
&=\mathbb{E}_{2n+1}^{\left( \alpha -1\right) }\left( x\right) -\mathbb{E}_{2n+1}^{\left(
\alpha -1\right) }\left( x-1\right) , \\
\sum_{k=1}^{n}\binom{2n}{2k-1}\mathbb{E}_{2k-1}^{\left( \alpha \right) }\left(
x\right)  &=\mathbb{E}_{2n}^{\left( \alpha -1\right) }\left( x\right)
-\mathbb{E}_{2n}^{\left( \alpha -1\right) }\left( x-1\right), \\
\sum_{k=0}^{n}\binom{2n+1}{2k}\mathbb{B}_{2k}^{\left( \alpha \right) }\left( x\right)
&=\left( n+\frac{1}{2}\right) \left( \mathbb{B}_{2n}^{\left( \alpha -1\right)
}\left( x\right) +\mathbb{B}_{2n}^{\left( \alpha -1\right) }\left( x-1\right) \right)
, \\
\sum_{k=1}^{n}\binom{2n}{2k-1}\mathbb{B}_{2k-1}^{\left( \alpha \right) }\left(
x\right)  &=n\left( \mathbb{B}_{2n-1}^{\left( \alpha -1\right) }\left( x\right)
+\mathbb{B}_{2n-1}^{\left( \alpha -1\right) }\left( x-1\right) \right) .
\end{align*}
Also, by multiplying each identity of the last identities by $x$, after that, replacing $x$ by $\textbf{A}+x$ and applying Lemma \ref{LL}, we deduce, after simplification:

\begin{align*}
\alpha \sum_{k=0}^{n}\binom{2n+1}{2k}E_{2k+1}^{\left( \alpha +1\right)
}\left( x\right) &=\left( \alpha +1\right) E_{2n+2}^{\left( \alpha \right)
}\left( x\right) -xE_{2n+1}^{\left( \alpha \right) }\left( x\right) -\left(
\alpha +1\right) E_{2n+2}^{\left( \alpha \right) }\left( x-1\right) \\ & \ \ +\left(
x-\alpha -1\right) E_{2n+1}^{\left( \alpha \right) }\left( x-1\right) , \\
\alpha \sum_{k=0}^{n}\binom{2n}{2k-1}E_{2k}^{\left( \alpha +1\right) }\left(
x\right) &=\left( \alpha +1\right) E_{2n+1}^{\left( \alpha \right) }\left( x\right)
-xE_{2n}^{\left( \alpha \right) }\left( x\right) -\left( \alpha +1\right)
E_{2n+1}^{\left( \alpha \right) }\left( x-1\right)\\ & \ \ +\left( x-\alpha
-1\right) E_{2n}^{\left( \alpha \right) }\left( x-1\right) ,
\end{align*}
\begin{align*}
\frac{\alpha }{n+\frac{1}{2}}\sum_{k=0}^{n}\binom{2n+1}{2k}B_{2k+1}^{\left(
\alpha +1\right) }\left( x\right) &=\left( \alpha +1\right) B_{2n+1}^{\left(
\alpha \right) }\left( x\right) -xB_{2n}^{\left( \alpha \right) }\left(
x\right) +\left( \alpha +1\right) B_{2n+1}^{\left( \alpha \right) }\left(
x-1\right) \\ & \ \ -\left( x+\alpha -1\right) B_{2n}^{\left( \alpha \right) }\left(
x-1\right) ,\\
\frac{\alpha }{n}\sum_{k=1}^{n}\binom{2n}{2k-1}B_{2k}^{\left( \alpha
+1\right) }\left( x\right) &=\left( \alpha +1\right) B_{2n}^{\left( \alpha
\right) }\left( x\right) -xB_{2n-1}^{\left( \alpha \right) }\left( x\right)
+\left( \alpha +1\right) B_{2n}^{\left( \alpha \right) }\left( x-1\right)
\\ & \ \ -\left( x+\alpha -1\right) B_{2n-1}^{\left( \alpha \right) }\left(
x-1\right).
\end{align*}
\end{example}

\begin{remark} Similarly, the most of the identities on Bernoulli polynomials given in \cite{Ruiz} can be generalized by Proposition \ref{P3}.
\end{remark}

\begin{proposition}\label{P}
Let $m,n$ be non-negative integers, let $\alpha$ be a real number and
let $(u_{k}), (v_{k})$ be two sequences of real numbers. If
\begin{equation*}
\sum_{k=0}^{n}U\left( n,k\right) \left( x+u_{k}\right)
^{k}=\sum_{k=0}^{n}V\left( n,k\right) \left( x+v_{k}\right) ^{k}
\end{equation*}%
for some real sequences $\left(U(n,k); 0\leq k\leq n\right)$ and $\left(V(m,k);0\leq k\leq m \right),$ then
\begin{align*}
\sum_{k=0}^{n}U\left( n,k\right) &f_{k}^{\left( \alpha \right) }\left(
x+u_{k}\right)  =\sum_{k=0}^{n}V\left( n,k\right) f_{k}^{\left( \alpha \right)
}\left( x+v_{k}\right) , \\
\sum_{k=0}^{n}U\left( n,k\right) &\left( f_{k+1}^{\left( \alpha \right)
}\left( x+u_{k}\right) -u_{k}f_{k}^{\left( \alpha \right) }\left( x+u_{k}\right) \right) \\
&=\sum_{k=0}^{n}V\left( n,k\right) \left( f_{k+1}^{\left( \alpha \right)
}\left( x+v_{k}\right) -v_{k}f_{k}^{\left( \alpha \right) }\left( x+v_{k}\right) \right).
\end{align*}
\end{proposition}

\begin{proof}
The first identity can be proved similarly as the first one of Proposition \ref{P3}.
So, by multiplying their two sides by $x,$ after that
replacing $x$ by $\textbf{A}+x$ we obtain%
\begin{equation*}
\sum_{k=0}^{n}U\left( n,k\right) \left( \textbf{A}+x\right) f_{k}^{\left( \alpha
-1\right) }\left( \textbf{A}+x+u_{k}\right) =\sum_{k=0}^{n}V\left( n,k\right) \left(
\textbf{A}+x\right) f_{k}^{\left( \alpha -1\right) }\left( \textbf{A}+x+v_{k}\right)
\end{equation*}%
which can written via Lemma \ref{LL}%
\begin{align*}
&\sum_{k=0}^{n}U\left( n,k\right) \left( f_{k+1}^{\left( \alpha \right)
}\left( x+u_{k}\right) +\left(\left( \alpha -1\right) x-u_{k}\right) f_{k}^{\left(
\alpha \right) }\left( x+u_{k}\right)  \right)  \\
&=\sum_{k=0}^{n}V\left( n,k\right) \left( f_{k+1}^{\left( \alpha \right)
}\left( x+v_{k}\right) +\left(\left( \alpha -1\right) x-v_{k}\right) f_{k}^{\left(
\alpha \right) }\left( x+v_{k}\right)  \right),
\end{align*}%
hence, use first the identity of this Proposition to simply this last one.
\end{proof}

\begin{example}
For any real numbers $\lambda, \alpha,$ the Ljunggren's identity \cite{lju}%
\begin{equation*}
\sum_{k=0}^{n}\binom{n}{k}\binom{\lambda }{k}y^{k}\left( x+y\right)
^{n-k}=\sum_{k=0}^{n}\binom{n}{k}\binom{\lambda +k}{k}y^{k}x^{n-k}
\end{equation*}%
implies, via Proposition \ref{P}, the following identities
\begin{align*}
\label{a} \sum_{k=0}^{n}\binom{n}{k}\binom{\lambda }{k}y^{k}f_{n-k}^{\left( \alpha
\right) }\left( x+y\right) &=\sum_{k=0}^{n}\binom{n}{k}\binom{\lambda +k}{k}%
y^{k}f_{n-k}^{\left( \alpha \right) }\left( x\right), \\
\sum_{k=0}^{n}\binom{n}{k}\binom{\lambda }{k}y^{k}\left( f_{n+1-k}^{\left(
\alpha \right) }\left( x+y\right) -yf_{n-k}^{\left( \alpha \right) }\left(
x+y\right) \right) &=\sum_{k=0}^{n}\binom{n}{k}\binom{\lambda +k}{k}%
y^{k}f_{n+1-k}^{\left( \alpha \right) }\left( x\right).
\end{align*}
\end{example}

\begin{example}
For any real numbers $\alpha, \beta, \gamma,$ the Munarini's identity \cite{mun}
\begin{equation*}
\sum_{k=0}^{n}\binom{\gamma }{k}\binom{\beta -\gamma +n}{n-k}y^{k}\left(
x+y\right) ^{n-k}=\sum_{k=0}^{n}\binom{\gamma}{n-k}\binom{\beta +k}{k}y^{k}x^{n-k}
\end{equation*}%
implies, via Proposition \ref{P}, the following identities
\begin{align*}
\sum_{k=0}^{n}\binom{\gamma }{k}\binom{\beta -\gamma +n}{n-k}%
y^{k}&f_{n-k}^{\left( \alpha \right) }\left( x+y\right) =\sum_{k=0}^{n}\binom{%
\gamma}{n-k}\binom{\beta +k}{k}y^{k}f_{n-k}^{\left( \alpha\right) }\left( x\right), \\
\sum_{k=0}^{n}\binom{\gamma }{k}\binom{\beta -\gamma +n}{n-k}y^{k}&\left(
f_{n+1-k}^{\left( \alpha \right) }\left( x+y\right) - yf_{n-k}^{\left( \alpha \right) }\left( x+y\right) \right) \\
&=\sum_{k=0}^{n}\binom{\gamma}{n-k}\binom{\beta +k}{k} y^{k} f_{n+1-k}^{\left( \alpha \right) }\left( x\right).
\end{align*}
\end{example}

\begin{example}
The Simons's identity \cite{sim}%
\begin{equation*}
\sum_{k=0}^{n}\left( -1\right) ^{n+k}\frac{\left( n+k\right) !}{\left(
n-k\right) !\left( k!\right) ^{2}}\left( x+1\right) ^{k}=\sum_{k=0}^{n}\frac{%
\left( n+k\right) !}{\left( n-k\right) !\left( k!\right) ^{2}}x^{k}
\end{equation*}%
implies, via Proposition \ref{P}, the following identities%
\begin{align*}
\sum_{k=0}^{n}\left( -1\right) ^{n+k}\frac{\left( n+k\right) !}{\left(
n-k\right) !\left( k!\right) ^{2}}f_{k}^{\left( \alpha \right) }\left(
x+1\right) &=\sum_{k=0}^{n}\frac{\left( n+k\right) !}{\left( n-k\right)
!\left( k!\right) ^{2}}f_{k}^{\left( \alpha \right) }\left( x\right), \\
\sum_{k=0}^{n}\left( -1\right) ^{n+k}\frac{\left( n+k\right) !}{\left(
n-k\right) !\left( k!\right) ^{2}}\left( f_{k+1}^{\left( \alpha \right)
}\left( x+1\right) -f_{k}^{\left( \alpha \right) }\left( x+1\right) \right)
&=\sum_{k=0}^{n}\frac{\left( n+k\right) !}{\left( n-k\right) !\left(
k!\right) ^{2}}f_{k+1}^{\left( \alpha \right) }\left( x\right).
\end{align*}
\end{example}

\begin{example}
The known identity \cite[Id. 3.17]{gou}%
\begin{equation*}
\sum_{k=0}^{n}\dbinom{n}{k}\dbinom{m}{k}\left( -1\right) ^{k}B_{k}\left(
x\right) =\sum_{k=0}^{n}\dbinom{n}{k}\dbinom{m+k}{k}B_{n-k}\left( -x\right)
\end{equation*}%
can be  written, on using the known identity $B_{n}(-x)=(-1)^{n}B_{n}(x+1)$, as
\begin{equation*}
\sum_{k=0}^{n}\dbinom{n}{k}\dbinom{m}{k}\left( -1\right) ^{k}B_{k}\left(
x\right) =\sum_{k=0}^{n}\dbinom{n}{k}\dbinom{m+k}{k}(-1)^{n-k}B_{n-k}\left( x+1\right)
\end{equation*}%
which now can be generalized, similarly to the proof of Proposition \ref{P3}, as
\begin{align*}
\sum_{k=0}^{n}\dbinom{n}{k}\dbinom{m}{k}&\left( -1\right) ^{k}B_{k}^{\left(
\alpha \right) }\left( x\right)  =\sum_{k=0}^{n}\dbinom{n}{k}%
\dbinom{m+k}{k}\left( -1\right) ^{n-k}B_{n-k}^{\left( \alpha \right) }\left( x+1\right) , \\
\sum_{k=0}^{n}\dbinom{n}{k}\dbinom{m}{k}&\left( -1\right) ^{k}
B_{k+1}^{\left( \alpha +1\right) }\left( x\right) \\
&=\sum_{k=0}^{n}\dbinom{n}{k}\dbinom{m+k}{k}\left( -1\right) ^{n-k}\left( B_{n+1-k}^{\left( \alpha
+1\right) }\left( x+1 \right)- B_{n-k}^{\left( \alpha
+1\right) }\left( x+1 \right) \right).
\end{align*}
\end{example}


\begin{thebibliography}{99}

\bibitem{appel}  P. Appell. Sur une classe de polyn\^{o}mes. \textit{ Ann Sci Ecole Norm Sup.} \textbf{9} (1880), 119--144.

\bibitem{ben}  A. Benyattou and M. Mihoubi. Curious congruences related to the Bell polynomials. \emph{Quaestiones Mathematicae.} Accepted paper, 2017.

\bibitem{car} L. Carlitz. A class of generating functions. \textit{SIAM J Math Anal.} \textbf{8} (1977), 518--532.

\bibitem {cre} A. Di Crescenzo and G. C. Rota, Sul calcolo umbrale, \textit{ Ricerche di Matematica} \textbf{43}  (1994), 129--162.

\bibitem {com} L. Comtet, \textit{Advanced combinatorics,} D. Reidel Publishing Co., Dordrecht, 1974.

\bibitem {gess} I. M. Gessel, Applications of the classical umbral calculus,
\textit{Algebra Universalis} \textbf{49} (2003), 397--434.

\bibitem {gou} H. W. Gould. \textit{Combinatorial Identities}. A Standardized Set of Tables Listing 500 Binomial Coefficient Summations, Revised Edition, published by the author, Morgantown, WV; 1972.

\bibitem {he} Y. He and W. Zhang. Some symmetric identities involving a sequence of polynomials. \textit{Electron J Combin.} \textbf{17} (2010), \#N7.

\bibitem {kan} M. Kaneko. A recurrence formula for the Bernoulli numbers. \textit{Proc Japan Acad Ser A Math Sci.} \textbf{71} (1995), 192--193.

\bibitem{lju} W. Ljunggren. Et element\ae rt bevis for en formel av AC Dixon. \textit{Norsk Mat Tidssrift.} \textbf{29} (1947), 35--38.

\bibitem{mac} P. A. MacMahon. \textit{Combinatory Analysis}. Two volumes (bound as one), Chelsea Publishing, New York; 1960.

\bibitem {mag} W. Magnus, F. Oberhettinger and R. P. Soni. \textit{Formulas and Theorems for the Special Functions of Mathematical Physics, third  ed.} Springer-Verlag,
 New York; 1966.

\bibitem {miho} M. Mihoubi. Bell polynomials and binomial type sequences,
\textit{Discrete Math.} \textbf{308} (2008), 2450--2459.

\bibitem {miho1} M. Mihoubi. The role of binomial type sequences in determination identities for Bell polynomials 111 (2013),
323.337.
\bibitem{mun} E. Munarini. Generalization of a binomial identity of Simons. \textit{Integers.} \textbf{5} (2005), \#A15.

\bibitem{rio} J. Riordan. \textit{An Introduction to Combinatorial Analysis}. John Wiley, New York, 1958.

\bibitem {Ruiz} R. Pita and J. V. Claudio De. Carlitz-type and other Bernoulli identities. \textit{J  Integer Seq.} \textbf{19} 2016, Article 16.1.8.

\bibitem {rom} S. Roman, \textit{The Umbral Calculus,} Academic Press, Orlando, FL, 1984.

\bibitem {sim} S. Simons. A curious identity. \textit{Math Gazette.} \textbf{85} (2001), 296--298.

\bibitem {sri} H. M. Srivastava. Some generalizations of Carlitz's theorem. \textit{Pacific J Math.} \textbf{85} (1979), 471--477.

\bibitem {sri1} H. M. Srivastava, I. L. Lavoie and R. Tremblay. A class of addition theorems. \textit{Canad Math Bull.} \textbf{26} (1983), 438--445.

\bibitem{xia} W. Xiaoying and Z. Wenpeng. Several new identities involving Euler and Bernoulli polynomials. \textit{Bull Math Soc Sci Math
Roumanie.} \textbf{59} (2016), 101--108.

\end{thebibliography}
\end{document}